\documentclass[11pt, reqno]{amsart}

\usepackage{ amssymb, amsmath, amsthm,  graphicx, psfrag}
\usepackage[usenames,dvipsnames]{color}
\usepackage{enumerate} 

\newcommand{\R}{{\mathbb R}}

\newcommand{\C}{{\mathbb C}}
\newcommand{\N}{{\mathbb N}}

\def\norm#1{\left\|#1\right\|}

\def\inpro#1{\left\langle #1 \right\rangle}
\def\Set#1{\left\{\,#1\,\right\}}\def\gp#1{\left(#1\right)}
\def\bk#1{\left[#1\right]}
\def\G{{\widetilde G}}

\def\bn{\begin{enumerate}}  
\def\en{\end{enumerate}}    
\def\bmt{\begin{matrix}} 
\def\emt{\end{matrix}}

\newtheorem{mythm}{Theorem}[section]
\newtheorem{theorem}{Theorem}[section]
\newtheorem{definition}[mythm]{Definition}

\newtheorem{corollary}[mythm]{Corollary}
\newtheorem{proposition}[mythm]{Proposition}
\newtheorem{lemma}[mythm]{Lemma}
\theoremstyle{remark}
\newtheorem{remark}[mythm]{Remark}
\theoremstyle{ob}
\newtheorem{ob}[mythm]{Observation}

\newcommand{\f}{\{f_i\}_{i=1}^k}

\newcommand{\F}{\mathcal{F}}

\newcommand{\Hn}{\mathcal{H}_n}

\newcommand{\diag}{\text{diag}}

\begin{document}

\title{ Diagram vectors and Tight Frame Scaling in Finite Dimensions$^*$}

\author{Martin S. Copenhaver}
\address{School of Mathematics, Georgia Institute of Technology}
\email{copenhaver@gatech.edu}

\author{Yeon Hyang Kim}
\address{Department of Mathematics, Central Michigan University}
\email{kim4y@cmich.edu}

\author{Cortney Logan}
\address{Department of Mathematics, Stonehill College}
\email{clogan@students.stonehill.edu}

\author{Kyanne Mayfield}
\address{Department of Mathematics, University of Portland}
\email{mayfield13@up.edu}

\author{Sivaram K. Narayan}
\address{Department of Mathematics, Central Michigan University}
\email{naray1sk@cmich.edu}

\author{Matthew J. Petro}
\address{Computer Aided Engineering, University of Wisconsin-Madison}
\email{petro@cae.wisc.edu}

\author{Jonathan Sheperd}
\address{Department of Mathematics, University of Notre Dame}
\email{jsheperd@nd.edu}

\thanks{*Reseach supported  by NSF-REU Grant DMS 08-51321. This work was done as a part of the REU program in  Summer 2011.}


\date{December 7, 2011.}


\begin{abstract}
We consider frames in a finite-dimensional Hilbert space $\Hn$ where frames are exactly the spanning sets of the vector space.  
The diagram vector of a vector in \(\R^2\) was previously  defined using polar coordinates and was used to characterize tight frames in \(\R^2\) in a geometric fashion. 
Reformulating the definition of a diagram vector in \(\R^2\) we provide a natural extension of this notion to \(\R^n\) and \(\C^n\). Using the diagram vectors we give a characterization of tight frames in \(\R^n\) or \(\C^n\).  Further we  provide a characterization of when a unit-norm frame in $\R^n$ or \(\C^n\)  
 can be scaled to a tight frame. This classification allows us to determine all scaling coefficients that make a unit-norm frame into a tight frame. 
\end{abstract}

\maketitle

\section{Introduction}\label{intro}
In recent years, new focus has been given to representation systems that are not a basis, but still admit stable decomposition and reconstruction algorithms.  The key notion in this regard is that of a frame. 
A frame in finite dimensions is a redundant set of vectors that span the  vector space. 
The study of frames began in 1952 with their introduction by Duffin and Schaeffer \cite{duffin} and has since been expanded by Daubechies \cite{daubechies} and others  \cite{BF03, classes, erasures, RS95}.   
  A basis is a  linearly independent spanning set.  
If $\f$ is an orthonormal baisis for a finite dimensional inner product space then each vector $f$ has a unique representation as 
$\displaystyle f = \sum_{i=1}^k \inpro{f, f_i} f_i$. 
 If a signal is represented as a vector and transmitted by sending the sequence of coefficients of its representation,  
 then using an orthonormal basis to analyze and later reconstruct the signal can be problematic. This is because the loss of any coefficient during transmission means that the original signal cannot be recovered. As a solution to this problem redundancy is introduced in frames so that it might be possible to reconstruct a signal if some coefficients are lost. A tight frame is a special case of a frame,  which has a reconstruction formula similar to  that of an orthonormal basis. Because of this simple formulation of reconstruction, tight frames are employed in a variety  of applications such as sampling, signal processing, filtering, smoothing, denoising, compression, image processing, and in other areas.  
 
 In \cite{physint}, the authors give necessary and sufficient conditions for the existence of a tight frame with a given sequence of norms, and provide a method of constructing a tight frame with such a sequence of norms. Here we address the question of when a tight frame exists such that the frame vectors point in specified directions. That is, given a sequence of unit vectors,  can we find a way of scaling each vector so that the resulting frame is tight? 
To answer this question, we begin by defining various notions that are mentioned above.  
A good introduction to frames in finite dimensions can be found in \cite{ffu}.

Let $I$ be a subset of $\N$. 
A \emph{frame} in a finite dimensional Hilbert space $\Hn$ 
 is a sequence of vectors $\{f_i\}_{i\in I}$ for which there exist constants $0<A\leq B<\infty$ such that for all $f \in\Hn$,
\[A\| f \|^2\leq\displaystyle\sum_{i\in I}|\langle f,f_i \rangle|^2 \leq B\|f\|^2.\]
When $A=B=\lambda$, $\{f_i\}_{i\in I}$ is called a \emph{$\lambda-$tight frame}. If $\lambda=1$ then the frame is called a \emph{Parseval frame}. A \emph{unit-norm frame} is a frame such that each vector in the frame has norm one. In a finite dimensional Hilbert space $\Hn$, a sequence of vectors is a frame if and only if it spans  $\Hn$.

Given a sequence of vectors  $\Set{f_i}_{i=1}^k$ in $\Hn$,  we define the \emph{analysis operator} to be the linear map $\theta: \Hn \rightarrow \ell^2(\Set{1, \cdots, k})$ defined by $(\theta f)(i)=\langle f, f_i \rangle$. The adjoint  $\theta^*\,:\, \ell^2(\Set{1, \cdots, k})\rightarrow \Hn$ is called the \emph{synthesis operator}. 
Using an orthonormal basis for  $\Hn$, the analysis operator associated with a sequence of vectors $\f$ can be written as  the $k\times n$ matrix
\[
\theta = \left[
\begin{array}{c c c }
\leftarrow & f_1^*& \rightarrow \\
& \vdots & \\
\leftarrow & f_k^*& \rightarrow 
\end{array}\right]
,\]
and the synthesis operator is the $n\times k$ matrix
\[\theta^* = \left[
\begin{array}{c c c }
\uparrow & & \uparrow \\
f_1 & \cdots & f_k\\
\downarrow & & \downarrow
\end{array}\right].
\]
The \emph{frame operator} $S$ of a sequence of vectors $\Set{f_i}_{i=1}^k$ (not necessarily a frame) is defined as $\theta^*\theta$. For all $f\in\Hn$,
\[
Sf=\theta^*\theta f= \displaystyle\sum_{i =1}^k \langle f,f_i\rangle f_i
.\]
For $\f\subset  \Hn$, the \emph{Gramian operator} $G$ is the $k\times k$ matrix defined by
\[
G=\theta\theta^*=(\langle f_i, f_j \rangle)_{i,j=1}^k.
\]
Given a sequence of vectors $\f$ in $\Hn$, it is known that the frame operator $S$ of the sequence has rank $n$ if and only if the sequence is a frame. The frame operator  $S= \lambda I_n$ if and only if $\f$ is a $\lambda-$tight frame. Also, $S = I_n$ if and only if $\f$ is a Parseval frame \cite{ffu}.

\section{Diagram Vectors in $\R^n  (\text{or } \C^n)$}

One of the simple tools we have for determining whether a frame for \(\mathbb{R}^2\) is tight or not is the notion of \emph{diagram vectors} \cite{ffu}. 
We express any vector $f$ in $\R^2$ using polar coordinates, 
$\displaystyle f = \bk{ \bmt  a \cos \theta \\ a \sin \theta \emt }$, where $\theta$ is the angle the vector makes with the positive $x$-axis. We  define the diagram vector associated with \(f\) by 
\[\tilde{f}= \bk{ \bmt a^2\cos{2\theta} \\ a^2\sin{2\theta} \emt }.\]
We observe that 
if $\displaystyle f = \bk{ \bmt  f(1) \\ f(2)  \emt }$, then 
\[ \tilde{f}= \bk{ \bmt  (f(1))^2- (f(2))^2 \\ 2f(1)f(2) \emt }. \]
This description of the diagram vector is useful algebraically, while the original definition is well-suited to geometric reasoning. The power of this notion comes from the next result, which follows  from the definition and the fact that a frame is tight if and only if its frame operator is a positive scalar multiple of the identity operator.

\begin{proposition} \cite{ffu} \label{Jprop:R2DVtight}
Let \(\{f_i\}_{i=1}^k\) be a sequence of vectors in \(\mathbb{R}^2\), not all of which are zero. Then \(\{f_i\}_{i=1}^k\) is a tight frame if and only if \(\sum_{i=1}^k\tilde{f_i}=0\).
\end{proposition}

The diagram of the sum of $\tilde{f}_i$ provides  a nice visual representation of the tight frame $\f$.
We seek to extend this definition to \(\R^n\) and $\C^n$.  In \(\R^2\) the condition that a frame is tight is equivalent to only two conditions on the components of the frame vectors, which allows us to define diagram vectors that are themselves in  \(\R^2\). However, in higher dimensions  it is impossible to reduce the tight frame condition to  \(n\) simple conditions, so there is little chance of preserving the useful geometric properties found in \(\R^2\). On the other hand, there is a property of the inner product of diagram vectors in $\R^2$ that we can preserve using our generalization of diagram vectors.

\begin{proposition}\label{Jprop:R2DVproduct}
If \(f,g\) are any vectors in \(\mathbb{R}^2\), then
\begin{equation*}\langle\tilde{f},\tilde{g}\rangle=2\langle f,g\rangle^2-\|f\|^2\|g\|^2.\end{equation*}
\end{proposition}

The following definition generalizes the notion of an associated diagram vector to a vector in  \(\R^n\) and allows us to prove analogues of Propositions \ref{Jprop:R2DVtight} and \ref{Jprop:R2DVproduct}. 
In the rest of the paper we denote $(f(i))^2$ as $f^2(i)$. 

\begin{definition}\label{Jdefn:RnDV}
For any vector \(f\in\mathbb{R}^n\), we define the diagram vector associated with \(f\), denoted \(\tilde{f}\), by
\begin{equation*}
\tilde{f}= 
\frac{1}{\sqrt{n-1}}
\begin{bmatrix}f^2(1)-f^2(2)\\  \vdots  \\ f^2(n-1)-f^2(n) \\  
\sqrt{2n}f(1)f(2) \\ \vdots  \\  \sqrt{2n}f(n-1)f(n)
 \end{bmatrix}\in\mathbb{R}^{n(n-1)},
\end{equation*}
where the difference of squares 
$f^2(i)- f^2(j)$ and the 
 product \(f(i)f(j)\)  occur exactly once for \(i < j, \ i = 1, 2, \cdots, n-1.\) 
 \end{definition}
Note that for \(n=2\) the above definition agrees with the standard notion of diagram vectors in \(\mathbb{R}^2\).

\begin{proposition}\label{Jprop:RnDVtight}
Let \(\{f_i\}_{i=1}^k\) be a sequence of vectors in \(\mathbb{R}^n\), not all of which are zero. Then \(\{f_i\}_{i=1}^k\) is a tight frame if and only if \(\sum_{i=1}^k\tilde{f_i}=0\). 
\end{proposition}

\begin{proof}
The sequence \(\{f_i\}_{i=1}^k\subset\mathbb{R}^n\) is a tight frame if and only if its frame operator \(S = \lambda I_n \), for some $\lambda > 0$.  
Consequently, we have
\begin{equation*}
\sum_{\ell=1}^k f_\ell(i)f_\ell(j)=
\begin{cases}
\lambda, \text{ if } i = j, \\
0, \text{ if } i \neq j.
\end{cases}
\end{equation*}
Equivalently, we get \(\sum_{i=1}^k\tilde{f_i}=0\). 
\end{proof}

\begin{proposition}\label{Jprop:RnDVproduct}
For any \(f,g\in\mathbb{R}^n\),  \(  (n-1) \langle\tilde{f},\tilde{g}\rangle=n\langle f,g\rangle^2-\|f\|^2\|g\|^2\).
\end{proposition}

\begin{proof}
For any vectors \(f\) and \(g\) in \(\mathbb{R}^n\), we have 
\begin{eqnarray*}
&&  n\langle f,g\rangle^2-\|f\|^2\|g\|^2 \\
 & = & n\left(\sum_{i=1}^n f(i)g(i)\right)^2-\left(\sum_{i=1}^n f^2(i)\right)\left(\sum_{i=1}^n g^2(i)\right) \\
& =& 2n\sum_{1\le i<j\le n} f(i)f(j)g(i)g(j)+n\sum_{i=1}^n f^2(i)g^2(i)  -\sum_{i=1}^n\sum_{j=1}^nf^2(i)g^2(j)\\
&=& 2n\sum_{1\le i<j\le n} f(i)f(j)g(i)g(j) + \sum_{1\le i<j\le n}\left(f^2(i)-f^2(j)\right)\left(g^2(i)-g^2(j)\right)\\
&=&  (n-1) \langle\tilde{f},\tilde{g}\rangle .
 \end{eqnarray*}
\end{proof}

\begin{remark}\label{rmk1}
From Proposition  \ref{Jprop:RnDVproduct}  it is immediate that 
if $\|f\|=1$ then $\| \tilde{f} \|= 1$. 
Suppose  $ S^{k}  := \Set{f \in \R^{k+1} \, :\, \norm{f} = 1 }$ is a sphere of   radius $1$ in $\R^{k+1}.$ The  assignment of a diagram vector $\tilde{f}$ to every unit vector in $\R^n$ can be thought of as an operator $D$ from 
$ S^{n-1} (1)$ into $ S^{n(n-1)-1} ( 1) $.  Moreover, $D$ is surjective if and only if $n=2$. 
For example, there is no unit vector $\displaystyle f = \bk{ \bmt a \\ b\\c \emt}  \in \R^3$ such that 
\[\tilde{f} = \frac{1}{\sqrt{2}} (0, 0, \sqrt{11}/4, 1, 1/2, 1/4) \in  S^5, \] since there are no $a, b, c \in \R$ such that 
$a^2 - b^2 = 0, \  \sqrt{6} ab =1, \  \sqrt{6} ac =1/2$, and $\sqrt{6} bc =1/4$. 
Also $D$ is not injective since $f$ and $-f$ have the same diagram vector. 
\end{remark}

For completeness, we present a definition of the associated diagram vector in \(\mathbb{C}^n\), which has  properties  similar  to Propositions   \ref{Jprop:RnDVtight}  and \ref{Jprop:RnDVproduct}. 
 
\begin{definition}\label{Jdefn:CnDV}
For any vector \(f\in\mathbb{C}^n\), we define the diagram vector associated with \(f\), denoted \(\tilde{f}\), by
\begin{equation*}
\tilde{f}= 
\frac{1}{\sqrt{n-1}}
\begin{bmatrix}f(1) \overline{f(1)}-f(2)\overline{f(2)} \\  \vdots  \\ f(n-1)\overline{f(n-1)}-f(n)\overline{f(n)} \\  
\sqrt{n}f(1) \overline{f(2)} \\ \sqrt{n} \overline{f(1)} f(2) \\ \vdots  \\  \sqrt{n}f(n-1)\overline{f(n)} 
\\ \sqrt{n} \overline{f(n-1)} f(n)
 \end{bmatrix}\in\mathbb{C}^{3n(n-1)/2},
\end{equation*}
where the difference of the form 
$f(i) \overline{f(i)} - f(j) \overline{f(j)}$ occurs exactly once for \(i < j, \ i = 1, 2, \cdots, n-1\)  and the 
 product of the form \(f(i)  \overline{f(j)} \)  occurs exactly once for \(i  \neq j.\) 
\end{definition}

By the same reasoning as in the proofs of Propositions    \ref{Jprop:RnDVtight}  and \ref{Jprop:RnDVproduct}, we have the following. 

\begin{proposition}\label{Jprop:CnDVtight}
Let \(\{f_i\}_{i=1}^k\) be a sequence of vectors in \(\mathbb{C}^n\), not all of which are zero. Then \(\{f_i\}_{i=1}^k\) is a tight frame  if and only if \(\sum_{i=1}^k\tilde{f_i}=0\). 
\end{proposition}

\begin{proposition}\label{Jprop:CnDVproduct}
For any \(f,g\in\mathbb{C}^n\),  \(  (n-1) \langle\tilde{f},\tilde{g}\rangle=n | \langle f,g\rangle |^2-\|f\|^2\|g\|^2\).
\end{proposition}

Note that we can also generalize the notion of  diagram vectors to  $\Hn$ by using isomorphisms. 
In the tight frame scaling problem we only consider frames in $\R^n (\text{or } \C^n)$.

\section{Tight Frame Scaling in $\R^n (\text{or } \C^n)$}

In this section, we first give a necessary condition for tight frame scaling in $\R^n$ and \(\C^n\). Though this condition is sufficient in \(\R^2\), we provide  an example to show the condition is not sufficient in \(\R^n\) when \(n\ge 3\). 
We then use the Gramian \(\G\) of associated diagram vectors  to give a necessary and sufficient condition for tight frame scaling in \(\R^n \) or \(\C^n\).
First we note that from Definition \ref{Jdefn:RnDV}, 
 for any real or complex scalar  $d$ and  a vector $ f$ in $ \R^n (\text{or } \C^n)$,  
\begin{equation} \label{diag_eq}
\widetilde{df} = |d|^2 \widetilde{ f}.\end{equation}
The next proposition shows that 
we can always choose the scaling constants for  a unit-norm frame \(\{f_i\}_{i=1}^k \subset \R^n (\text{or } \C^n) \)  to  be nonnegative  numbers. 

\begin{proposition}\label{Jprop:TFSPformulations}
Let \(\{f_i\}_{i=1}^k\) be a unit-norm frame for \(\mathbb{R}^n (\text{or } \C^n)\). 
 There exist real or complex scalars  \(d_1,d_2,\ldots,d_k\) for which \(\{d_if_i\}_{i=1}^k\) is a tight frame for \(\mathbb{R}^n (\text{or } \C^n)\) if and only if  
there exist 
nonnegative   numbers  \(c_1,c_2,\ldots,c_k\) such that \(\{c_if_i\}_{i=1}^k\) is a tight frame for \(\mathbb{R}^n (\text{or } \C^n)\) .
\end{proposition}

The proof of Proposition \ref{Jprop:TFSPformulations} follows from selecting $c_i = |d_i|$ and Equation (\ref{diag_eq}). 
The following lemma provides a useful condition for tight frame scaling in $\R^n$.

\begin{lemma}\label{unit_tight}
Let \(\{f_i\}_{i=1}^k\) be a unit-norm frame for \(\mathbb{R}^n (\text{or } \C^n)\).  If there exist 
nonnegative   numbers  \(c_1,c_2,\ldots,c_k\) such that \(\{c_if_i\}_{i=1}^k\) is a $\lambda$-tight frame for 
some \(\lambda \neq 0 \), then  
$$ c_1^2 + \cdots + c_k^2 = \lambda n.$$
\end{lemma}

\begin{proof}
If \(\{c_if_i\}_{i=1}^k\) is a $\lambda$-tight frame for \(\mathbb{R}^n  (\text{or } \C^n) \), then since the frame operator for \(\{c_if_i\}_{i=1}^k\)  is $\lambda I_n$, we have the trace
$$  \sum_{j=1}^n \sum_{i=1}^k c_i^2 f_i^2(j) =\lambda  n. $$
Since \(\{f_i\}_{i=1}^k\) is a unit-norm frame, for each $i$, 
$$ f_i^2(1) + \cdots + f_i^2(n) =1,  $$
which implies that  
\begin{equation*}
 c_1^2 + \cdots + c_k^2 = \lambda n.
 \end{equation*}
\end{proof}

We use Lemma \ref{unit_tight} to provide a necessary condition for tight frame scaling in $\R^n  (\text{or } \C^n)$. 

\begin{theorem}\label{Jprop:TFSPinHn}
Let \(\{f_i\}_{i=1}^k\) be a unit-norm frame for \( \R^n  (\text{or } \C^n) \). If there exist positive numbers 
\(c_1,c_2,\ldots,c_k\) such that \(\{c_i f_i\}_{i=1}^k\) is a tight frame for \( \R^n  (\text{or } \C^n) \), then there is no unit vector \(f\in \R^n  (\text{or } \C^n)\) such that \(|\langle f,f_i\rangle|\ge 1/\sqrt{n}\) for all \(i =  1, 2,  \cdots,  k\) and \(|\langle f,f_i\rangle|>1/\sqrt{n}\) for at least one \(i\). 
\end{theorem}

\begin{proof}
Let  \(\{c_i f_i\}_{i=1}^k\) be a  $\lambda$-tight frame. 
Suppose there is a unit vector \(f\in \R^n  (\text{or } \C^n) \) such that \(|\langle f,f_i\rangle|\ge 1/\sqrt{n}\) for all \(i =  1, 2,  \cdots,  k\) and \(|\langle f,f_i\rangle|>1/\sqrt{n}\) for at least one \(i\). 
Since \(\{c_i f_i\}_{i=1}^k\) is a  $\lambda$-tight frame and $ \norm{f} =1$, by the assumption, we have 
\begin{equation*}
\lambda=\sum_{i=1}^k  c_i^2|\langle f,f_i\rangle|^2>\frac{1}{n}\sum_{i=1}^k c_i^2, 
\end{equation*}
which contradicts Lemma \ref{unit_tight}.
\end{proof}

We show in Theorem \ref{Jthm:TFSPinR2}  the necessary condition of   Theorem \ref{Jprop:TFSPinHn} is also sufficient in \(\R^2\). 
For an counterexample of the converse of Theorem \ref{Jprop:TFSPinHn}, we  now construct a unit-norm frame  \( \f \)which can not be scaled to be  tight and  which satisfies the property that 
\begin{center}
 ``(P)  there is no unit vector \( f \in S^2 \) such that  \( |\langle f,f_i\rangle|\ge 1/\sqrt{3}\) for all \(i =  1, 2,  \cdots,  k\)."
 \end{center}
First, we consider the  unit-norm frame \(\F= \Set{f_1, \cdots,  f_5},\)
where \[
f_1 =  \gp{ \begin{matrix} 1\\0\\0\end{matrix}}, 
f_2 =  \gp{ \begin{matrix} 0\\1\\0 \end{matrix}}, 
f_3 =  \gp{ \begin{matrix} 0\\0\\1 \end{matrix}}, 
f_4 =  \gp{ \begin{matrix} 1/\sqrt{2} \\1/\sqrt{2} \\0 \end{matrix}}, 
f_5 =  \gp{ \begin{matrix} 1/\sqrt{2} \\-1/\sqrt{2} \\0 \end{matrix}}.
\]
Let
\[ C_1 =\Set{f  \in S^2 \,:\, |\langle f,f_i\rangle|\ge 1/\sqrt{3}, \  i =1, 2, 3 }
\]
 and 
\[ C_2 =\Set{f  \in S^2 \,:\, |\langle f,f_i\rangle|\ge 1/\sqrt{3}, \  i =3, 4, 5 }
.\] 
Then we have 
\[C_1 =\Set{  \gp{ \begin{matrix} \pm 1/\sqrt{3} \\   \pm 1/\sqrt{3} \\  \pm 1/\sqrt{3} \end{matrix} } }, \quad 
C_2  =\Set{  \gp{ \begin{matrix} \pm \sqrt{2/3} \\  0  \\  \pm 1/\sqrt{3} \end{matrix} }, \ 
\gp{ \begin{matrix} 0 \\ \pm \sqrt{2/3}  \\  \pm 1/\sqrt{3} \end{matrix} } } ,\]
both consisting of eight points. 
Since $C_1 \bigcap C_2 = \emptyset$, $\F$ satisfies the property (P).   
The frame $\F$ can be scaled to be tight as it is the union of two orthonormal bases of $\R^3$.  
By perturbing $f_1, f_2$ and $f_3$, we obtain a unit-nom frame which can not be scaled to be  tight and  satisfies the property (P). 
To this end, we choose $v >0 $ close to zero and let 
$ u = \sqrt{1-2v^2}. $ 
We consider 
the  unit-norm frame  \(\F_v = \Set{f_1', \cdots,  f_5'},\)
where \[
f_1'=  \gp{ \begin{matrix} u \\ v \\ v  \end{matrix}},  \
f_2'=  \gp{ \begin{matrix} v \\ u\\ v \end{matrix}},  \
f_3' =  \gp{ \begin{matrix} v \\ v  \\ u \end{matrix}},\
f_4'= f_4, \
f_5'=f_5.
\]
Let
\[ C_1' =\Set{f  \in S^2 \,:\, |\langle f,f_i' \rangle|\ge 1/\sqrt{3}, \  i = 1, 2, 3 }
\]
 and 
\[ C_2' =\Set{f  \in S^2 \,:\, |\langle f,f_i' \rangle|\ge 1/\sqrt{3}, \  i =3, 4, 5 }
.\] 
Then $C_1'$ and $C_2'$ are subsets of a closed neighborhood in $S^2$ of $C_1$ and $C_2$, respectively. 
For   $v$ sufficiently small,  we have 
 $C_1' \bigcap C_2' = \emptyset$, which implies that $\F_v$ satisfies the property (P).   
 We now show that $\F_v$  can not be scaled to be tight. 
 For fixed $v$, suppose that 
 there exist positive numbers 
\(c_1, \ldots,c_5\) such that  for any $f\in \R^3$, 
\( \sum_{i=1}^5 \inpro{ c_i f_i , f } c_i f_i = f. \) 
Comparing coefficients, this 
  implies that 
 \( v^2 c_1^2 + uv c_2^2 + uv c_3^2 =0. \) 
 This contradicts the condition that $u, v, c_1, c_2$ and $c_3$ are positive.  
 The following figure shows the parallelepiped  determined by the perturbed frame $\F_v$ along with $C_1'$ and $C_2'$ when $v = 1/8$. 
 
 \begin{figure}
\begin{center}
\includegraphics[width=0.7\linewidth]{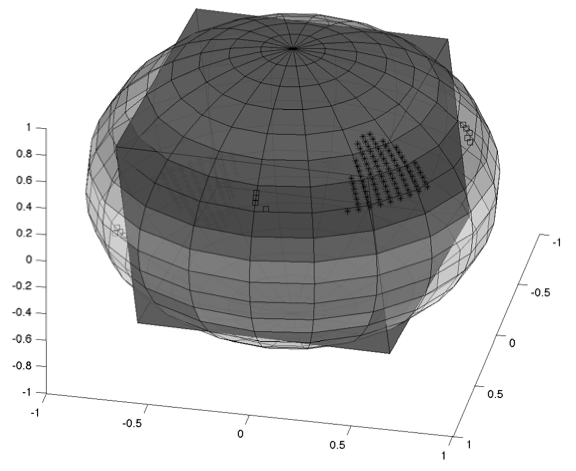} 
 \caption{\em  The symbol * represents the elements of $C_1'$ and the symbol $\Box$ represents the elements of $C_2'$}
\end{center}
\end{figure}

A similar construction can be given in \(\R^n\) for \(n \ge 4\).

\begin{remark} 
For a geometric description of the conclusion of Theorem \ref{Jprop:TFSPinHn},  
we define for  $f\in \R^n$, 
$$ cone(f) : =  \Set{ g \in \R^n \,:\,  | \inpro{g, f} | \ge  1/ \sqrt{n}}. $$
Then the conclusion of Theorem \ref{Jprop:TFSPinHn} is equivalent to 
$$ S^{n-1}(1) \bigcap \gp{\bigcap_{i=1}^k   cone(f_i)} \bigcap \gp{ \bigcup_{i=1}^k cone(f_i)^\circ} = \emptyset,  $$
where $cone(f_i)^\circ$ is the interior of $cone(f_i)$. 
\end{remark}

\begin{corollary}\label{cor2}
Let  \(\{f_i\}_{i=1}^k\) be  a unit-norm frame for \( \R^n\).  
If  $ \displaystyle \bigcap_{i=1}^k   cone(f_i)^\circ  \neq  \emptyset$, then 
there do not exist nonzero scalars \(c_1,c_2,\ldots,c_k\) such that \(\{c_i f_i\}_{i=1}^k\) is a tight frame for \(\mathbb{R}^n\). 
\end{corollary}

As  a consequence of Corollary \ref{cor2}, for example,  two vectors in $\R^2$ from the same open quadrant or three vectors in 
$\R^3$ from the same open octant  cannot be scaled to a tight frame. 


We use Gramian of associated diagram vectors of the unit-norm frame to provide a necessary and sufficient condition for the tight frame scaling in $\R^n (\text{or } \C^n)$. 
For $\f\subset  \R^n (\text{or } \C^n)$, the Gramian operator $G$ is the $k\times k$ matrix defined by
\[
G= \theta \theta^* =  \gp{ \inpro{f_i, f_j}}_{i,j=1}^k ,
\]
where $\theta$ and $\theta^*$ are the corresponding analysis and synthesis operator of $\f$.  
 We note that the Gramian matrix is a Hermitian and positive semidefinite matrix.

\begin{remark}\label{rmk2}
Let $A$ be a Hermitian positive semidefinite matrix and $x \in \C^n$. Then 
$x^* A x =0$ if and only if $Ax=0$ (\cite{hj}, p.400).  Moreover, if $A$ is positive definite then $x^* A x = 0$ if and only if $x=0$. 
\end{remark}

Let $\G = \gp{ \inpro{\tilde{f_i},\tilde{f_j}}}_{i,j=1}^k $ be the Gramian associated to the diagram vectors 
 $\Set{\tilde{f}_i}_{i=1}^k$. 
We  note that if  $\f\subset \R^n (\text{or } \C^n)$ is a set of unit-norm vectors, then by  Proposition 
\ref{Jprop:RnDVproduct}, we have $(n-1)\G_{ij} = n |G_{ij}|^2 - 1$, which implies that \(\G\) is a real matrix. 
We also have the following characterization of tight frames using the Gramian $\G$.

\begin{proposition}\label{Gpro}
Let  \(\{f_i\}_{i=1}^k\) be  a unit-norm frame for \( \R^n (\text{or } \C^n)\) and  $c_1, \cdots, c_k$ be nonnegative numbers, which are not all zero. 
Let $\G$ be the Gramian associated to the diagram vectors $\Set{\tilde{f}_i}_{i=1}^k$. 
Then  $\{c_i f_i\}_{i=1}^k $ is a tight frame  for $\R^n (\text{or } \C^n)$ if and only if 
 $f = \left[
\begin{array}{c}
c_1^2\\
\vdots\\
c^2_k
\end{array}\right] $ belongs to the null space of $\G$. 
\end{proposition}
\begin{proof}
The set  $\{c_i f_i\}_{i=1}^k \subset \R^n (\text{or } \C^n)$ is a tight frame 
 if and only if 
  $\sum_{i=1}^k  \widetilde{c_if_i} = \sum_{i=1}^k  c_i^2 \tilde{f_i}=0 $ 
  if and only if 
  $\displaystyle  \left\langle \sum_{i=1}^k  c_i^2 \tilde{f_i}  ,  \sum_{i=1}^k  c_i^2 \tilde{f_i}\right\rangle =f^T \G f  =0 $. 
  From Remark \ref{rmk2}, since $\G$ is positive semidefinite, we conclude that $\G f =0$. 
\end{proof}

\begin{ob} \label{ob1}
From Remark \ref{rmk2}, if  $\G$ is invertible  then $f^T \G f = 0$ if and only if $f = 0$. 
From Proposition \ref{Gpro} we conclude that if $\G$ is invertible, then there do not exist 
nonnegative numbers \(c_1, c_2, \cdots, c_k\) which are not all zero so that $\{c_i f_i\}_{i=1}^k $ is  a tight frame. 
Therefore, it is enough to consider the case where $\G$ is not invertible. 
\end{ob}

For a given subset $K$ in $\R^n$, let $\text{conv} (K)$ be the convex hull of $K$. 
We then have the following result on linear inequalities. 

\begin{lemma}[\cite{cheney}, p.19]\label{convthm}
Let $K$ be a compact subset of $\R^n$. There exists $g \in \R^n$  such that $\inpro{g, f} >0$ for all $f \in K$ if and only if $0 \notin  {\rm conv}(K)$.
\end{lemma}

We now prove the main theorem for  tight frame scaling problem in $\R^n (\text{or } \C^n)$.  

\begin{theorem}\label{mthm:tfs}
Let $\f$ be a unit-norm frame for $\R^n  (\text{or } \C^n)$ and let 
$\G $ be the Gramian associated to the diagram vectors $\Set{\tilde{f_i}}_{i=1}^k$.
Suppose $\G$ is not invertible. Let  $\Set{v_1,\ldots,v_l}$ be any basis for $ null (\G)$ and $r_i :=
\left[
\begin{array}{c}
v_1 (i)\\
\vdots\\
v_l (i)
\end{array}\right]$ for $i=1,\ldots,k$. 
Then there exist $c_i>0, \ i =1, \cdots, k$  so that $\{c_if_i\}_{i=1}^k$ is a tight frame if and only if $0 \not\in  {\rm conv}  \Set{r_1,\ldots,r_k}$.
\end{theorem}

\begin{proof}
Suppose  there exist  $c_i>0, \ i =1, \cdots, k$  so that $\{c_i f_i\}_{i=1}^k$ is a tight frame.  
Then by Proposition \ref{Gpro}, this is equivalent to $ \G f  =0 $, where  $f = \left[
\begin{array}{c}
c_1^2\\
\vdots\\
c^2_k
\end{array}\right]$.
Since $\Set{v_1,\ldots,v_l}$ is a basis for $ null (\G)$,  
$$ f = \sum_{i=1}^l y(i) v_i = 
\left[
 \begin{array}{c}
\inpro{y, r_1}\\
\vdots\\
\inpro{y, r_k}
\end{array}\right],
$$ for some $y \in \R^l$.  
Since $ \inpro{y, r_i} = c_i^2 >0$ for $i = 1, 2, \cdots, k$ by considering $K = \Set{ r_1,\ldots,r_k }$ in  Lemma \ref{convthm}, 
this is equivalent to $0 \not\in {\rm conv} \Set{r_1,\ldots,r_k}$.
This completes the proof. 
\end{proof}

Using Theorem \ref{mthm:tfs} we can explicitly determine when there exists a solution to the tight frame scaling problem. 
We determine all scaling coefficients that make a unit-norm frame into a tight frame using the set 
\[ \gp{\bigcap_{j=1}^k H_i}^\circ, \]
where  \( H_i = \Set{y \in \R^l \,:\,   \inpro{y, r_i} \ge 0}, i = 1\cdots, k.  \)
For example, given a unit-norm frame 
\(  \Set{ \gp{ 
\begin{matrix}
1\\ 0
\end{matrix}}, 
\gp{\begin{matrix}
0\\ 1
\end{matrix}}, 
\gp{\begin{matrix}
-1\\ 0
\end{matrix}
} } 
\) in \(\R^2\), and a basis 
\(  \Set{ \gp{ 
\begin{matrix}
-1\\0\\1
\end{matrix}}, 
\gp{\begin{matrix}
1\\1\\0
\end{matrix}
} } 
\) 
for  \(null(\G)\), we get  
\(   r_1 = \gp{ 
\begin{matrix}
-1\\1
\end{matrix}}, 
 r_2 = \gp{ 
\begin{matrix}
0\\1
\end{matrix}},  \text{ and }
 r_3 = \gp{ 
\begin{matrix}
1\\0
\end{matrix}}\). 
Then 
\[ \gp{\bigcap_{j=1}^3 H_i}^\circ = \Set{(x,y) \,:\, x >0,\ y>0,\ y>x}:\]
\[\includegraphics[width=0.3\linewidth]{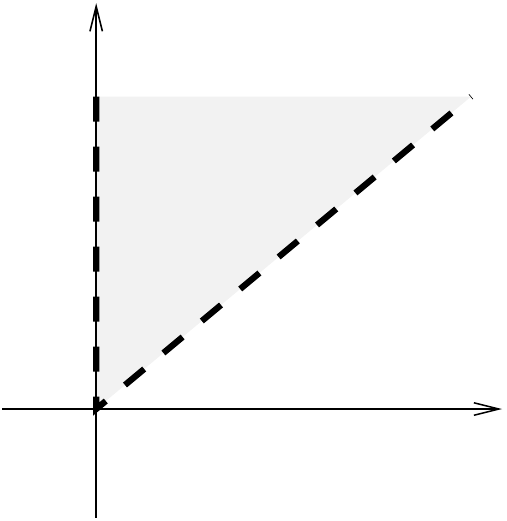}   \]
For any \( c \in  \gp{\bigcap_{j=1}^3 H_i}^\circ\), let 
\( a = Bc\), where \( B =\bk{\begin{matrix}
-1 & 1 \\
0 & 1\\
1 & 0
\end{matrix}} \) is the matrix generated from the basis of 
\(null(\G)\). 
Then Theorem \ref{mthm:tfs} implies that 
\( \sqrt{a(1)}, \sqrt{a(2)}\), and \( \sqrt{a(3)}\) are tight frame scaling coefficients. 

In general, how can we compute the coefficients for the tight frame scaling when such a scaling exists?  
We use the perceptron algorithm which provides a set of scaling coefficients for a given unit-norm frame in $\R^n$ or $\C^n$ to be a tight frame. The perceptron algorithm solves the problem \( Bc>0 \), where $B$ is the matrix generated from the basis of 
\(null(\G)\) \cite{perceptron, Rosenblatt}. Then the set \( \sqrt{a(1)}, \cdots, \sqrt{a(k)}\) are tight frame scaling coefficients, where \(a = By\). 

There is an equivalent formulation of the tight frame scaling problem as a matrix equation using the Gramian. 
Let $\f$ be a unit-norm frame for $\R^n$. There exist scalars $c_i>0, \ i =1, \cdots, k$  so that $\{c_if_i\}_{i=1}^k$ is a tight frame if and only if there exists a $C = \diag(c_1,\ldots,c_k)$ with $c_i > 0$ for all $i$ so that $(\theta^* C^*) (C \theta) = \lambda I = (C\theta)^* (C\theta)$ for some $\lambda > 0$. 
Multiplying the equation on the left by $\theta$  and on the right by $\theta^*$ yields   $\theta \theta^* C^* C \theta \theta^* = \lambda \theta \theta^* $.
Since $G = \theta \theta^*$, we have the following result. 

\begin{proposition}\label{thmM}
Let $\f$ be a unit-norm frame for $\R^n$ and let 
$G$ be the Gramian associated to $\f$. 
There exist scalars $c_i>0, \ i =1, \cdots, k$  so that $\{c_if_i\}_{i=1}^k$ is a \(\lambda\)-tight frame if and only if  
there exist a $D:= \diag(d_1,\ldots,d_k)$ with $d_i > 0$ for all $i=1,\ldots,k$ so that $GDG = \lambda G$.
\end{proposition}

\section{Tight Frame Scaling  in $\R^2$}

In Section 3, we used the Gramian of the unit-norm frame to characterize the tight frame scaling in $\R^n$ or $\C^n$.  
In the next theorem we prove that the necessary condition in Theorem \ref{Jprop:TFSPinHn} is also sufficient for tight frame scaling in $\R^2$. 

\begin{theorem}\label{Jthm:TFSPinR2}
Let \(\{f_i\}_{i=1}^k\) be a unit-norm frame for \( \R^2\). Then the following are equivalent:
\bn[1. ]
\item There exist positive  numbers  \(c_1,c_2,\ldots,c_k\) such that \(\{c_if_i\}_{i=1}^k\) is a tight frame for \(\mathbb{R}^2\).

\item There is no unit vector \(f\in\mathbb{R}^2\) such that \(|\langle f,f_i\rangle |\ge 1/\sqrt{2}\) for all \(i=1,2,\ldots,k\) and \(|\langle f,f_i\rangle |>1/\sqrt{2}\) for at least one \(i\).
\en
\end{theorem} 

We have  proved the forward direction of Theorem \ref{Jthm:TFSPinR2}. 
However, to prove the reverse direction  for a given unit-norm frame  \(\{f_i\}_{i=1}^k\) in \( \R^2\), we need to consider the following set, which is associated with the diagram vectors of the frame:
\[ U :=\left\{\sum_{i=1}^k c_i^2\tilde{f_i} \,: \, c_i  \neq 0, \  i=1, 2, \cdots, k  \right\}. \]

For any unit-norm vectors \(\{w_i\}_{i=1}^k\) in \( \R^2\), it is easy to see that 
\( W :=\left\{\sum_{i=1}^k a_i  w_i \,: \, a_i >0, \,\,  i=1, 2, \cdots, k  \right\} \) is a convex set.  
From Remark \ref{rmk1}, since the diagram operator in $\R^2$ preserves unit vectors, we obtain the following. 

\begin{lemma}\label{Jlemma:R2convex}
If \(\{f_i\}_{i=1}^k\) is  a unit-norm frame for \( \R^2\), then  
\[ U :=\left\{\sum_{i=1}^k c_i^2\tilde{f_i} \,: \, c_i  \neq 0, \  i=1, 2, \cdots, k\right\} \]
is a convex set.
\end{lemma}

As a consequence of  Proposition \ref{Jprop:RnDVproduct} and  the surjectivity of the diagram operator for $n=2$,  we obtain the following  equivalent formulation of the second statement of Theorem \ref{Jthm:TFSPinR2}, which provides a connection to the set $U$. 

\begin{lemma} \label{conv_pro}
Let \(\{f_i\}_{i=1}^k\) be a unit-norm frame for \( \R^2\).  Then the following are equivalent:
\bn[1. ]
\item There is no unit vector \(f\in\mathbb{R}^2\) such that \(|\langle f,f_i\rangle |\ge 1/\sqrt{2}\) for all \(i=1,2,\ldots,k\) and \(|\langle f,f_i\rangle |>1/\sqrt{2}\) for at least one \(i\).
\item There is no unit vector \(g \in\mathbb{R}^2\) such that \(  \langle g,  \tilde{f}_i\rangle  \ge 0 \) for all \(i=1,2,\ldots,k\) and \( \langle g ,\tilde{f}_i\rangle  > 0\) for at least one \(i\).
\en
\end{lemma} 

\begin{proof}
($\Rightarrow$)
Suppose there exists a unit vector \(g \in\mathbb{R}^2\) such that \(  \langle g,  \tilde{f}_i\rangle  \ge 0 \) for all \(i=1,2,\ldots,k\) and \( \langle g ,\tilde{f}_i\rangle  > 0\) for at least one \(i\). Since the diagram operator $D\,:\, S^1(1) \to S^1(1)$ is onto, there exists $f \in S^1(1)$ such that $\tilde{f} =g$.  Then, by Proposition \ref{Jprop:RnDVproduct},  this unit vector $f$ satisfies that \(|\langle f,f_i\rangle |\ge 1/\sqrt{2}\) for all \(i=1,2,\ldots,k\) and \(|\langle f,f_i\rangle |>1/\sqrt{2}\) for at least one \(i\). 

($\Leftarrow$)
Suppose there  exists a unit vector \(f\in\mathbb{R}^2\) such that \(|\langle f,f_i\rangle |\ge 1/\sqrt{2}\) for all \(i=1,2,\ldots,k\) and \(|\langle f,f_i\rangle |>1/\sqrt{2}\) for at least one \(i\).
Then, by Proposition \ref{Jprop:RnDVproduct},  the unit vector $g =\tilde{f}$ in $\R^2$ satisfies that  \(  \langle g,  \tilde{f}_i\rangle  \ge 0 \) for all \(i=1,2,\ldots,k\) and \( \langle g ,\tilde{f}_i\rangle  > 0\) for at least one \(i\).
\end{proof}

In the next two propositions, we present a few  properties of the set $U$. 

\begin{proposition}\label{U1}
Let  \(\{f_i\}_{i=1}^k\) be  a unit-norm frame for \( \R^2\). 
If   
\( 0 \notin U :=\left\{\sum_{i=1}^k c_i^2\tilde{f_i} \,: \, c_i  \neq 0, \  i=1, 2, \cdots, k \right\} \), 
then \(\{\tilde{f}_i\}_{i=1}^k\) is also  a unit-norm frame for \( \R^2\). 
\end{proposition}

\begin{proof} By Propostion \ref{Jprop:RnDVproduct},  \(\{\tilde{f}_i\}_{i=1}^k\) is a set of  unit-norm vectors. Suppose that \(\{\tilde{f_i}\}_{i=1}^k\) does not span \(\mathbb{R}^2\).  Then 
all the diagram vectors are collinear and furthermore, by Lemma \ref{Jlemma:R2convex}, 
 all vectors must be positive scalar multiples of one another because \(0\notin U\). But it is clear from the definition of the diagram vector associated with a vector that \(\tilde{f}=\tilde{g}\) for some \(f,g \in\mathbb{R}^2\) if and only if \(f=\pm g \). This implies that all the vectors in \(\{f_i\}_{i=1}^k\) are collinear, so they do not span \(\mathbb{R}^2\), which is a contradiction.
\end{proof}

\begin{proposition}
Let  \(\{f_i\}_{i=1}^k\) be  a unit-norm frame for \( \R^2\). 
If   
\( 0 \notin U :=\left\{\sum_{i=1}^k c_i^2\tilde{f_i} \,: \, c_i  \neq 0, \  i=1, 2, \cdots, k \right\} \), 
then $U$ is an open set. 
\end{proposition}

\begin{proof}
Let \(T:\mathbb{R}^k\to\mathbb{R}^2\) be the linear map defined by \(T(v)=\sum_{i=1}^k v(i)\tilde{f_i}\) for all \(v\in\mathbb{R}^k\). By Proposition \ref{U1},   \(T\) is surjective. Since $T$ is surjective and continuous, by the open mapping theorem,  $ T\gp{ (0,\infty)^k} = U$ is an open set. 
\end{proof}

\begin{lemma}[\cite{hyper} p.26] \label{hyperplane}
If $w$ is a boundary point of  a closed convex set $V$, then there exists at least one supporting hyperplane of $V$ passing through $w$.  
\end{lemma}

Using lemmas \ref{Jlemma:R2convex},  \ref{conv_pro}, \ref{hyperplane} and Proposition \ref{U1}  we now prove the sufficency part of Theorem \ref{Jthm:TFSPinR2}.

\begin{proof}[Proof of Theorem \ref{Jthm:TFSPinR2}]
Since the forward direction has been proved in Theorem \ref{Jprop:TFSPinHn},  
we prove the contrapositive of the reverse direction. 
Assume that there do not exist nonzero scalars \(c_1,c_2,\ldots,c_k\) such that \(\{c_i f_i\}_{i=1}^k\) is a tight frame for \(\mathbb{R}^2\). 
Since $\displaystyle  \sum_{i=1}^k  \widetilde{ c_i f_i} = \sum_{i=1}^k c_i^2\tilde{f_i}  $, this is equivalent to saying 
\[0\notin U:= \left\{\sum_{i=1}^k c_i^2\tilde{f_i} \,: \, c_i  \neq 0, \  i=1, 2, \cdots, k \right\} .\]
Let $\bar{U}$ be the closure of $U$. 
Since $\bar{U}$ is a closed convex set in $\R^n$ and the zero is a point on the boundary of $U$, by 
Lemma \ref{hyperplane}, there exists a supporting hyperplane of $\bar{U}$ passing through $0$.  
Let $g \in \R^2$ be a unit normal  vector for the supporting hyperplane such that for all $h \in \bar{U}$,  $\inpro{g , h} \ge 0.$
We claim that  \(  \langle g,  \tilde{f}_i\rangle  \ge 0 \) for all \(i=1,2,\ldots,k\) and \( \langle g ,\tilde{f}_i\rangle  > 0\) for at least one \(i\).
If the first part of the claim does not hold for some \(i\), then making \(c_i\) very large compared to the other scalars would yield an element \(h\) of \(U\) with \(\langle g,h\rangle<0\). 
If the second part of the claim does not hold, then \(\left\langle g,\tilde{f_i}\right\rangle=0\) for all  \(i=1, 2, \cdots, k\), which contradicts Proposition \ref{U1}. Hence there  exists a unit vector \(g \in\mathbb{R}^2\) such that \(  \langle g,  \tilde{f}_i\rangle  \ge 0 \) for all \(i=1,2,\ldots,k\) and \( \langle g ,\tilde{f}_i\rangle  > 0\) for at least one \(i\).
Using Lemma \ref{conv_pro}, we have completed the proof of the theorem. 
\end{proof}

The following  corollary is a consequence of Theorem \ref{Jthm:TFSPinR2}.
\begin{corollary}\label{cor1}
Let  \(\{f_i\}_{i=1}^k\) be  a unit-norm frame for \( \R^2\).  
If  $ \displaystyle \bigcap_{i=1}^k   cone(f_i)  = \emptyset$, then 
there exist positive  numbers  \(c_1,c_2,\ldots,c_k\) such that \(\{c_if_i\}_{i=1}^k\) is a tight frame for \(\mathbb{R}^2\).
\end{corollary}

Note that  Corollary  \ref{cor1} provides  a new way to construct a tight frame which has more irregular angle distribution  than real Harmonic frames \cite{ffu}. \\

Next we present an algorithm to produce the tight frame scaling constants in $\R^2$ using the results from this section.  
 We generate a sequence of scalars producing a tight frame 
when the  unit-norm frame $\f$ for \(\mathbb{R}^2\)  satisfies the property that 
\begin{center}
 ``(Q)  there is no unit vector \( f \in\mathbb{R}^2\) such that  \(|\langle f,f_i\rangle|\ge 1/\sqrt{2}\) for all \(i =  1, 2,  \cdots,  k\) and \(|\langle f,f_i\rangle|>1/\sqrt{2}\) for at least one \(i\)."
 \end{center}
This method works by finding sets of \(2\) and \(3\) vectors that can be scaled to produce a tight frame.  
If \(\{f_i\}_{i=1}^2\) is a unit-norm frame for \(\mathbb{R}^2\) with the property (Q), then by  Theorem  \ref{Jthm:TFSPinR2} and Corollary \ref{cor2}, $f_1$ and $f_2$ are perpendicular to each other.  Thus we take  \(c_1=c_2=1\). 
If \(\{f_i\}_{i=1}^3\) is a unit-norm frame for \(\mathbb{R}^2\) with the property (Q),   we let \(J:\mathbb{R}^2\to\mathbb{R}^2\) be a linear operator that rotates each vector counter-clockwise by \(\pi/2\) radians, so that in particular \(\langle Jf,f\rangle=0\) for each \(f\in\mathbb{R}^2\). 
Since the property (Q) implies that 
 for each unit vector \(y\in\mathbb{R}^2\), \(\langle y,\tilde{f_i}\rangle\) takes either all zero values or both positive and negative values and  \(\langle J \tilde{f_1},\tilde{f_1}\rangle=0\), 
$$  \langle J \tilde{f_1},\tilde{f_2}\rangle =\langle J \tilde{f_1},\tilde{f_3}\rangle = 0, \text{ or}  $$
$$ \langle J \tilde{f_1},\tilde{f_2}\rangle \cdot \langle J \tilde{f_1},\tilde{f_3}\rangle < 0.  $$
If  both \(\langle J \tilde{f_1},\tilde{f_2}\rangle\) and \(\langle J \tilde{f_1},\tilde{f_3}\rangle\) are zero, then we let \(c_2=c_3=1\). Otherwise assume without loss of generality that \(\langle J \tilde{f_1},\tilde{f_2}\rangle>0\) and \(\langle J \tilde{f_1},\tilde{f_3}\rangle<0\), and let \(c_2=\sqrt{-\langle J \tilde{f_1},\tilde{f_3}\rangle}\) and \(c_3=\sqrt{\langle J \tilde{f_1},\tilde{f_2}\rangle}\). 
In either case,  let \(c_1=\sqrt{-\langle\tilde{f_1},c_2^2\tilde{f_2}+c_3^2\tilde{f_3}\rangle}\).  
Then we get
$$  \left\langle J \tilde{f_1},\sum_{i=1}^3\widetilde{c_if_i}\right\rangle =
\left\langle\tilde{f_1},\sum_{i=1}^3\widetilde{c_if_i}\right\rangle =0,
$$
which requires that the sum of the scaled diagram vectors is zero.

\begin{proposition}\label{Jthm:R2coeffs}
If \(\{f_i\}_{i=1}^k\) is a unit-norm frame for \(\mathbb{R}^2\) with the property (Q), then 
 every vector in the frame is a member of some \(2\)- or \(3\)-subset of the frame with the property (Q).
 \end{proposition}

\begin{proof}
Let \(a\in I:=\{1,2,\ldots,k\}\) be any index.
If there exists of \(b \in I \) for which \(\langle\tilde{f_a},\tilde{f_b}\rangle=-1\), then \(\{f_a,f_b\}\) is a frame satisfying the given condition. If, on the other hand, there is no such index \(b\), we claim that there exist \(b,c\in I\) so that \(\{f_a,f_b,f_c\}\) is a frame satisfying the property (Q).
Suppose that it is false. Then for every pair \(b,c\in I \backslash \Set{a} \),  there is a unit vector \(y\in\mathbb{R}^2\) such that \(\langle y,\tilde{f_i}\rangle\ge 0\) for all \(i=a,b,c\) and \(\langle y,\tilde{f_i}\rangle>0\) for some \(i\). 
For each index \(i = 1, 2, \cdots, k\), we define  \(H_i :=\{y\in\mathbb{R}^2:\|y\|=1,\langle y,\tilde{f_i}\rangle\ge 0\}\). 
Then our assumption implies that for every pair \(b,c\in I \backslash \Set{a} \), we have  \(H_a\cap H_b\cap H_c\neq\emptyset\). 
Since $H_i$ is a compact set, by  Helly's Theorem (\cite{cheney}, p. 19), we have 
  \(\bigcap_{i=1}^k H_i\neq \emptyset\). 
 Therefore there is a unit vector \(y\in\mathbb{R}^2\) such that \(\langle y,\tilde{f_i}\rangle\ge 0\) for all \(i = 1, 2, \cdots, k \). Note that \(\langle y,\tilde{f_i}\rangle \) cannot be zero for all \(i = 1, 2, \cdots, k \) since we are assuming \(-\tilde{f_a}\) is not the diagram vector of any vector in the frame, so for some \(i\), \(\langle y,\tilde{f_i}\rangle>0\). This contradicts the property (Q) of the  unit-norm frame $\f$. 
 \end{proof}
 
Using Proposition \ref{Jthm:R2coeffs} together with Corollary \ref{cor1} and \ref{cor2}, we provide the algorithm to generate the squares of scaling coefficients, $C_1^2, \cdots, C_k^2$, for a unit-norm frame $\f$ of \(\mathbb{R}^2\) with the property (Q) to be tight.
  \bn[1. ]
 \item Set $I:=\{1,2,\ldots,k\}$ and  $C_i^2 := 0$ for all $i\in I$.  
 \item Define $I_1 := \Set{a \in I \, :\, \inpro{f_a, f_b} =0 \text{ for some } b \in I}, \ 
 I_2 := I \backslash I_1$.
  \item For each $a \in I_1, \ b \in I_1\backslash \Set{a}$, 
 if $\inpro{f_a, f_b} =0$ then \\
 set $C_a^2 := C_a^2 +1$ and $C_b^2 := C_b^2+ 1$.  
 \item For each $a \in I_2, \ b \in I \backslash \Set{a}, \ c \in I \backslash \Set{a, b}$,
 if $ {\rm cone}(f_a) \cap {\rm cone}(f_b) \cap {\rm cone}(f_c)  \neq \emptyset$, then 
 compute  $c_a, c_b$ and $c_c$ using the method described for sets of 3 vectors in this section and \\
 set $C_a^2 := C_a^2 + c_a^2, \ C_b^2 := C_b^2+ c_b^2, \ C_c^2 := C_c^2 + c_c^2$. 

 \en

\section*{Acknowlegement}
Copenhaver, Logan, Mayfield, Narayan, and Sheperd were supported by the NSF-REU Grant DMS 08-51321.
 Kim was supported by the Central Michigan University ORSP Early Career Investigator
(ECI) grant \#C61373.

\bibliographystyle{amsplain}

\bibliography{References}

%
%

\end{document}